\documentclass[a4paper, 11pt, reqno]{amsart}
\usepackage{graphicx,color}
\usepackage{latexsym}
\usepackage{graphicx}
\usepackage[english]{babel}
\usepackage[utf8]{inputenc}
\usepackage{wrapfig}
\usepackage{amsthm,amsmath, amssymb,amsfonts,esint}
\usepackage{setspace}
\usepackage{hyperref}%para ir de uma referência a sua localização o texto

 \usepackage{ esint }
\usepackage{graphicx}%
\usepackage{multirow}%
\usepackage{amsmath,amssymb,amsfonts}%
\usepackage{amsthm}%
\usepackage{mathrsfs}%
\usepackage[title]{appendix}%
\usepackage{xcolor}%
\usepackage{textcomp}%
\usepackage{manyfoot}%
\usepackage{booktabs}%
\usepackage{algorithm}%
\usepackage{algorithmicx}%
\usepackage{algpseudocode}%
\usepackage{listings}%
\newtheorem{theorem}{Theorem}[section]

\newtheorem{corolary}{Corollary}[section]
\newtheorem{proposition}[theorem]{Proposition}
\newtheorem{lemma}[theorem]{Lemma}
\newtheorem{definition}[theorem]{Definition}
\newtheorem{corollary}[theorem]{Corollary}
\newtheorem{rem}[theorem]{Remark}

\title{Band Width Estimates and Rigidity of Manifolds with Negative Curvature}
\author[T.~Cruz]{Tiarlos Cruz}
\address{Institute of Mathematics, Federal University of Alagoas, 57072-970, 
	\linebreak\indent Maceió-AL, Brazil}
\email{\href{mailto: cicero.cruz@im.ufal.br}{cicero.cruz@im.ufal.br}}



\begin{document}
\begin{abstract}
We establish optimal Lipschitz lower bounds for proper smooth functions on three-dimensional Riemannian manifolds with Ricci curvature bounded below by negative constants, yielding a new family of width estimates for Riemannian bands using Gromov's $\mu$-bubble method, together with rigidity statements characterizing the equality case.  One of the novelties of our approach lies in its ability to handle higher-genus boundary components, revealing a precise interplay between the width, the area of boundary surfaces, and the underlying topology.  Finally, for a complete noncompact three-manifold $M$ with bounded geometry and scalar curvature $R_g\ge -6$, whose $H_2(M,\mathbb{Z})$ contains no spherical or toroidal classes, we prove a sharp lower bound for the boundary area. In the equality case, the manifold is shown to be isometric to an infinite hyperbolic band.
\end{abstract}
\maketitle

\section{Introduction}

A recurring theme in Riemannian geometry is understanding how lower curvature bounds constrain the global size and geometry of a manifold. This interplay manifests differently across the curvature spectrum.  Positive curvature typically forces compactness and diameter control under positive Ricci curvature assumptions.  Negative curvature, by contrast, allows substantially greater flexibility: complete manifolds with negative sectional curvature may have infinite volume, complicated geometry at infinity, and no \emph{a priori} diameter bounds.
This dichotomy raises the natural question: 
given a lower bound on the curvature, how does it constrain the distance between boundary components? And if negative curvature alone cannot bound the size of the manifold, what additional structure does it require?

To formalize the problem, recall that a  Riemannian band is a compact connected manifold $M$ with boundary decomposed into two nonempty unions of connected components $\partial_{-}M$ and $\partial_{+}M$. The \emph{width} of $(M,g)$ is defined as
$$
\text { width}(M, g)=\operatorname{dist}_g\left(\partial_{+} M, \partial_{-} M\right), 
$$
where $\operatorname{dist}_g$ is the distance on $M$ with respect to the metric $g.$

In the positive curvature setting, the answer to the above question is by now well understood, thanks largely to Gromov's $\mu$-bubble technique. Using the minimal hypersurface techniques of Schoen-Yau \cite{SY}, together with a symmetrization argument, Gromov \cite{G2} proved that if $\mathbb{T}^2 \times [-1,1]$ carries a metric with scalar curvature $R_g \geq 6$, then
$
\operatorname{width}(M,g) \leq 2\pi/3.
$
It is remarkable that scalar curvature alone can impose such a bound on the distance, since distance estimates typically require stronger curvature assumptions. 
We also mention the works of Zeidler \cite[Theorem 1.4]{Z} and Cecchini \cite[Theorem D]{C}, who established higher-dimensional analogues using Dirac operators.

Gromov’s problem can also be formulated in terms of sectional curvature, as it is motivated by the case where $\mathbb{T}^{2}\times[-1,1]$ admits an isometric immersion in $\mathbb{S}^3$. More recently, resolving a conjecture of Gromov originally formulated under the stronger assumption of sectional curvature bounds, Zhu \cite{Zhu2} proved that if $(M^3,g)$ is a smooth overtorical band with Ricci curvature $\operatorname{Ric}_g \geq 2$, then the width of $M$ is bounded above by $\pi/2$. He also established a rigidity statement: in the case of equality, $M$ is isometric to a quotient of $\mathbb{R}^2 \times (-\varepsilon,\varepsilon)$ endowed with a  doubly warped product metric.

More recently, attention has turned to the negative curvature regime, where much less is known about width estimates. In contrast to the positive curvature case, negative curvature alone does not impose topological obstructions: every compact manifold of dimension at least three admits a metric of negative scalar curvature by the works of Aubin \cite{A}, Blank–Kalka \cite{BK}, and Kazdan–Warner \cite{KW}, while Gao–Yau \cite{GY} and Lohkamp \cite{L} established the corresponding result in the negative Ricci curvature setting.  Moreover, Cruz–Vitório \cite{CV} showed that every compact manifold admits a metric with negative scalar curvature and minimal boundary.

The following result by Gromov \cite[Section 4]{G3} provides a first width estimate in the negative scalar curvature setting (stated here in dimension three).

\begin{theorem}[Gromov]\label{refgromov} Let $M=\mathbb{T}^{2}\times[0,1]$ be a $3$-dimensional Riemannian manifold endowed with a metric $g$ satisfying  $R_g\geq-6$ and assume that the mean curvature of $\partial_+M=\mathbb{T}^{2}\times\{1\}$  satisfies $$H_g^{\partial_+M} \ge2\coth(\pi/2),$$
where $H^{\partial_+M}_g$ denotes the mean curvature with respect to the outward-pointing unit normal. 
Then 
$$
\text { width}(M, g)\leq \frac{\pi}{3}.
$$
\end{theorem}

Theorem \ref{refgromov} shows that, under a mean curvature barrier on one boundary component, 
the width of a torical band is uniformly bounded above. To extend such estimates to more general topologies, we introduce the following class of proper functions.

\begin{equation*}
\mathcal{F}^{(\mathfrak g)}_M = \left\{  \phi \in C^\infty(M, [0, 1)) \,\middle|\, 
\begin{array}{l}
\phi \text{ is proper surjective  with }\phi^{-1}(0)=\partial M   \\
\text{ and } \phi^*([\text{pt}]) \text{ satisfies the $(\star)_{\mathfrak g}$-condition}
\end{array}
\right\}.
\end{equation*}
Here $\phi^*([\text{pt}])$ denotes the pull-back of the homology class $[\text{pt}].$ The $(\star)_{\mathfrak g}$-condition means that $\phi^*([\text{pt}])$ cannot be represented by a 2-cycle all of whose connected components have genus at most ${\mathfrak g}.$ 

Our first main result provides an optimal Lipschitz lower bound for a class of proper functions on $3$-manifolds with $\operatorname{Ric} \ge -2$. 

\begin{theorem}\label{teoA}
Let $(M^{3}, g)$ be an orientable noncompact Riemannian manifold with compact boundary $\partial M$. Suppose $\operatorname{Ric}_g \geq -2$, $\mathcal{F}_M^{(0)}\neq\emptyset$, and that the mean curvature of $\partial M$ with respect to the outward unit normal satisfies $H^{\partial M}_g\geq 2\coth(\frac{\pi}{2})$. Then for every $\phi\in\mathcal{F}_M^{(0)}$
    $$
    \operatorname{Lip}\phi \geq \frac{4}{\pi}.
    $$
    If equality holds, then,  up to scaling, $(M^{3}, g)$ is isometric to $M_0 / \Gamma$ for a lattice $\Gamma$ of $\mathbb{R}^2$ and $\phi$ coincides, up to scaling and translation, with the signed distance function on $M_0 / \Gamma$. 
Here $M_0 =  \mathbb{R}^2\times[0, \tfrac{\pi}{4}) $ is equipped with the doubly warped product metric  
$$g = dt^2 +  \cosh^2\left(\frac{\pi}{4}-t\right)ds_1^2 +  \sinh^2\left(\frac{\pi}{4}-t\right) ds_2^2,$$
with $(s_1,s_2)$ standard coordinates on $\mathbb{R}^2$.
 \end{theorem}

We obtain the following analogue of Gromov’s width estimate
for manifolds satisfying $\operatorname{Ric}\ge -2$.

\begin{corolary}\label{princ_00}
   Let $M=\mathbb T^2\times [-1,1]$ carry a metric $g$ with  $Ric_g \geq -2$ and $H^{\partial_- M}_g\geq 2\coth(\tfrac{\pi}{2})$ with respect to the outward unit normal.
    Then
$$\textrm{width}(M,g)\leq
\frac{\pi}{4}.$$
\end{corolary}

The estimate in Corollary \ref{princ_00} was first obtained by Chai-Sun \cite{CS}  for torical bands (see their Theorem B.1 with $\kappa = -1,$ $t_- = \pi/4,$  $t_+ = \pi/2$, and  $\eta(t)=2\coth(2t)$). In contrast, %Our Theorem \ref{teoB} treats their case to arbitrary half-open bands with $\mathcal{F}_M^{(0)}\neq\emptyset$, and  
our Corollary \ref{princ_00} holds without assuming mean curvature bounds on both boundary components. This type of estimate is part of a broader investigation into rigidity and Ricci/scalar curvature bounds on bands. For more related results in the context of Riemannian bands and open incomplete manifolds, we refer to the recent work of Hirsch-Kazaras-Khuri-Zhang \cite{HKKZ} and references therein.

We also consider another natural perspective, which is to study the behavior of the width of Riemannian bands  whose boundary has genus greater than one. In this setting, the width estimates depend on the area of one boundary component. Here, the area of a boundary component $\Sigma$ with respect to the induced metric will be denoted by $|\Sigma|.$

 \begin{theorem}\label{teoB} 
 Let $(M^3,g)$ be an orientable noncompact Riemannian manifold with compact connected boundary. Assume that
 $\operatorname{Ric}_g \geq -2,$ and that $\mathcal{F}^{(\mathfrak{g}-1)}_M$ is nonempty for some integer $\mathfrak{g}\ge2$. Given $c\in(0,2),$ suppose that $H^{\partial M}_g \geq \sqrt{4-2c}\,\coth(\pi/2)$ with respect to the outward unit normal.
 If 
 $$|\partial M| \leq \frac{4\pi(\mathfrak{g}-1)}{c},$$ 
 then every $\phi\in \mathcal F_M^{(\mathfrak{g}-1)}$ satisfies
    $$
    \operatorname{Lip}\phi \geq
    \frac{2\sqrt{4-2c}}{\pi}.
    $$
    \end{theorem}

We present the following consequence, which holds for manifolds for which $H_2(M; \mathbb Z)$ contains no classes represented by closed surfaces of genus  $g(\Sigma
)$ less than two. 

\begin{corollary}\label{princ_0}
Let $\Sigma$ be a closed surface of genus $g(\Sigma) \geq 2$ and let $M = \Sigma \times [-1,1]$ be endowed with a Riemannian metric $g$ such that $\operatorname{Ric}_g \geq -2$. Given  $c \in (0,2),$ assume that $H^{\partial_- M}_g \geq \sqrt{4-2c}\,\coth\!\left(\frac{\pi}{2}\right)$  with respect to the outward-pointing unit normal, and the area of $\partial_- M$ satisfies
$$
|\partial_- M| \leq \frac{4\pi\bigl(g(\Sigma)-1\bigr)}{c}.
$$
Then the width of $(M,g)$ is bounded by
$$\textrm{width}(M,g)\leq \frac{\pi}{2\sqrt{4-2c}}.$$
\end{corollary}

A key ingredient in the proof of our results is Gromov’s notion of a $\mu$-bubble \cite{Grom0, Grom}, which can be viewed as a minimizer of a weighted area-type functional. More specifically, for a Riemannian manifold $(M^n,g)$ and a given function $h$, a $\mu$-bubble is defined as a minimizer, or more generally, a critical point, of the functional
\begin{equation*}
\Omega \mapsto \mathcal{H}^{n-1}(\partial \Omega) - \int_{\Omega} h,
\end{equation*}
where $\mathcal{H}^{n-1}$ denotes the induced $(n-1)$-dimensional Hausdorff measure, and $\Omega \subset M$ varies within an appropriate class of subsets, for instance, Caccioppoli sets. The $\mu$-bubble technique generalizes and localizes the stable minimal hypersurface method, and is effective in the setting of noncompact manifolds. The $\mu$-bubble technique has found many applications across a range of problems, see, for example, \cite{CRZ, CS, CL0, CL, LY, Mazet, Wang, Zhu2, Zhu}, and references therein.

In the scalar curvature setting, one can obtain an analogous estimate for manifolds with higher-genus boundary components.  Specifically, let $M = \Sigma \times [-1,1]$ admit a metric with scalar curvature $R_g \ge -6$. Suppose that there exists a constant $c \in (0,2)$ such that the mean curvature of $\partial_- M$ with respect to the outward unit normal satisfies $H^{\partial_- M}_g\geq c\cdot\coth(\tfrac{\pi}{2}).$ If, in addition, 
$$|\partial_-M|\leq \frac{4\pi(\mathfrak{g}-1)}{3-\frac{3}{4}c^2},$$
then (see Section~\ref{scalar}) we can show that 
$$\textrm{width}(M,g)\leq\frac{\pi}{2c}.$$ 
Although the resulting width estimate is weaker than the corresponding Ricci curvature estimate, it should be viewed as the scalar curvature analog of the Corollary \ref{princ_0}. 
This suggests a sharp limiting inequality as $c$ goes to $0$, where the width estimate degenerates and the area of higher-genus surfaces becomes the relevant geometric quantity. 
 
To obtain a rigidity statement in the scalar curvature setting, additional
control on the geometry at infinity is required. Indeed, in the noncompact
setting one must prevent degeneration phenomena, which naturally leads to the
bounded geometry assumption. More precisely, we assume that there exist
constants $\Lambda>0$ and $i_0>0$ such that
$$
|\mathrm{Rm}_g|\le \Lambda
\qquad\text{and}\qquad
\operatorname{inj}_M\ge i_0>0,
$$
where $\mathrm{Rm}_g$  denotes the Riemann curvature tensor of 
$g$ and $\operatorname{inj}_M$ denotes the injectivity radius of 
$M$%.

Define
$$ \mathfrak g_0 := \min\left\{ g(\Sigma):\Sigma\subset M\mbox{ connected, closed, and homologically nontrivial} \right\}. $$
\begin{theorem}\label{princ}
   Let $(M^{3}, g)$ be an orientable complete Riemannian manifold with connected and mean-convex boundary $\partial M$. Assume that $(M,g)$ has  bounded geometry,  scalar curvature $R_g \geq -6$ and that  $H_2(M; \mathbb Z)$  is non‑trivial and does not contain spherical and toroidal classes. 
Then
     \begin{equation}\label{ineq_thm}
         |\partial M| \geq \frac{4\pi}{3} (\mathfrak g_0-1).
     \end{equation}
 The equality holds if and only if $(M, g)$ is isometrically covered  by a cylinder $\partial M\times[0,+\infty),$ where  $\partial M$ has constant Gauss curvature equal to $-3.$ Moreover, $\mathfrak g_0$ coincides with the genus of $\partial M.$
\end{theorem}

We have 
\begin{corollary}
Let $\Sigma$ be a closed surface of genus $g(\Sigma)\geq 2$, and let
$M=\Sigma\times [0,+\infty)$ be endowed with a Riemannian metric $g$ with 
$
R_g\geq -6
$
and mean-convex boundary. Assume that $(M,g)$ has bounded geometry and that
$$
|\partial M| \leq \frac{4\pi\bigl(g(\Sigma)-1\bigr)}{3}.
$$
Then $(M,g)$ is isometrically covered by the Riemannian product
$$
\bigl(\Sigma\times[0,+\infty),\,dt^2+g_\Sigma\bigr),
$$
where $g_\Sigma$ is a metric of constant Gauss curvature $-3$ on $\Sigma$.
\end{corollary}

One of the main difficulties in proving Theorem \ref{princ} is to guarantee that the sequence of approximating $\mu$-bubbles converges to a well-defined limiting surface instead of diverging to infinity. This requires a careful control mechanism to retain the sequence within a fixed region and ensure compactness.  This argument is inspired by the ideas in \cite{CCE,CEM,Liu,Zhu}.

We also mention the related cuspidal rigidity results of R\"ade \cite[Theorem 2.15]{R}, who considered manifolds manifolds $(M^n,g)$ with $3 \le n \le 7$ satisfying $R_g \ge -n(n-1)$ and 
$H_{\pm \partial M} \ge \pm (n-1)$ (where $H_{\pm \partial M}$ denotes the mean curvature of $\pm\partial M$). 
In a complementary direction, Hao--Hu--Liu--Shi \cite{HHLS} studied complete, noncompact 
manifolds $(M^n,g)$ under the additional assumptions that $\partial M$ is connected, incompressible in $M$, 
and belongs to the class $\mathcal{C}_{\mathrm{deg}}$,\footnote{A closed, connected 
manifold $X$ belongs to $\mathcal{C}_{\mathrm{deg}}$ if it is aspherical and no compact 
manifold $X'$ admitting a nonzero-degree map to $X$ can carry a metric of positive scalar curvature.} and proved a cuspidal boundary rigidity theorem.   However, these results rely heavily on the specific topology $T^2\times[a,b]$. Nevertheless, a systematic study of width estimates under negative curvature assumptions, especially for manifolds whose boundary components have genus greater than one, remains largely open.

\medskip

\noindent\textbf{Outline of the paper.}  Section \ref{preliminaries} contains the necessary preliminaries on $\mu$-bubbles and prescribed mean curvature hypersurfaces. In Section \ref{teoab} we prove Theorem \ref{teoA} and \ref{teoB}, including the higher-genus extensions of the width estimates under lower Ricci curvature bounds. Section \ref{scalar} is devoted to the proof of Theorem \ref{princ},   where we establish a sharp lower bound for the boundary area under scalar curvature and bounded geometry assumptions.

\medskip

\noindent\textbf{Acknowledgments:} 
The author expresses sincere gratitude to Lucas Ambrozio for insightful conversations during the Summer Program 2025 at IMPA, and to Ivaldo Nunes and Almir Santos for helpful discussions. He thanks Xiaoxiang Chai and Yukai Sun for valuable comments and suggestions as well. Part of this work was carried out during a visit to the Instituto de Matemática Pura e Aplicada (IMPA), whose hospitality is gratefully acknowledged. The author was partially supported by CNPq, Brazil, under grant numbers 307419/2022-3, 408834/2023-4, 444531/2024-6, 400078/2025-2, 403770/2024-6 and Coordenação de Aperfeiçoamento de Pessoal de Nível Superior (CAPES/ MATH-AMSUD 88887.985521/2024-00), Brazil.

\section{The Prescribed mean curvature functional}\label{preliminaries}
Let $(M^3, g)$ be a complete, noncompact, orientable Riemannian 3-manifold with connected compact boundary $\partial M$.   Assume that there exists $\phi\in \mathcal{F}^{(\mathfrak g)}_M.$ 
Given a smooth function $h:[0,T)\to\mathbb{R}$, we define the functional 
\begin{align}\label{modified_functional}
\mu^h(\Omega)=\mathcal{H}^2(\partial^*\Omega)-\int_{M}\chi_{M\setminus\Omega}\cdot h\circ\phi\,\mathrm{d}\mathcal{H}^3
\end{align}
on the class of Caccioppoli\footnote{For a reference on Caccioppoli sets, see \cite{Giu}.
} sets $\Omega\subset M$ with reduced boundary $\partial^*\Omega$, subject to the condition
\begin{align*}
M\setminus\Omega\Subset\phi^{-1}([0,T)).
\end{align*}
This modified functional was introduced by Zhu \cite{ZhuC}.

Let $\{\Omega(t)\}_{t\in (-\varepsilon,\varepsilon)}$ be a smooth one-parameter family of regions with $\Omega(0)=\Omega$, and let $\varphi$ denote the normal speed at $t=0$. The first variation of $\mu^h$ is given by
\begin{equation*}
\frac{d}{dt} \bigg|_{t=0} \mu^h(\Omega(t)) =\int_{\partial \Omega} (H_{\partial \Omega} + h \circ \phi) \varphi,
\end{equation*}
where $H_{\partial \Omega}$ denotes the mean curvature of $\partial\Omega$  as
the boundary of $\Omega$ with respect to the outward-pointing unit normal vector field $\nu$.

Since the variation  is arbitrary, it follows that $\Omega$ is critical for $\mu^h$ if and only if $H_{\partial \Omega} = -h \circ \phi$ along $\partial \Omega$. In this case, $\Omega$ is called a $\mu$-bubble, and $\partial\Omega$ is a hypersurface with prescribed mean curvature $-h \circ \phi$.

For a $\mu$-bubble $\Omega$, the second variation of $\mu^h$ is given by
\begin{equation}\label{second}
\frac{d^2}{dt^2}\bigg|_{t=0} \mu^h(\Omega(t)) = \int_{\partial \Omega} |\nabla \varphi|^2 - ( \operatorname{Ric}_g(\nu,\nu)+|A|^2 + \langle \nabla (h \circ \phi), \nu \rangle ) \varphi^2,
\end{equation}
where $\operatorname{Ric}_g$ is the Ricci curvature of $(M^3,g)$ and $A$ is the second fundamental form of $\partial\Omega$.

We say that a $\mu$-bubble $\Omega$ is stable if the second variation is nonnegative for every smooth normal variation, i.e., if the right-hand side of \eqref{second} is nonnegative for all smooth functions $\varphi$.

A   $\mu$-bubble $\hat \Omega$ is a local minimizer if it minimizes the localized energy among nearby competitors:
 for every Caccioppoli set $\Omega$ such that $\Omega \mathbin{\Delta} \hat \Omega \Subset M$ (i.e., the symmetric difference is compactly contained in $M$) and for every open set $U$ with compact closure satisfying $\Omega \mathbin{\Delta} \hat{\Omega} \Subset U$, we have
$$
\mu^h_U(\Omega) \ge \mu^h_U(\hat{\Omega}),
$$
where
$$
\mu^h_U(\Omega) = \mathcal{H}^2(\partial^*\Omega \cap U) - \int_{\Omega \cap U} (h \circ \phi) \, d\mathcal{H}^3.
$$

Recall that a \emph{band} is a connected compact manifold $M$ together with a decomposition of its boundary:
\begin{equation*}
\partial M = \partial_{-}M \,\dot{\cup}\, \partial_{+}M,
\end{equation*}
where $\partial_{-}M$ and $\partial_{+}M$ are nonempty unions of connected components of $\partial M$.

\begin{definition}
Let $(M, g, \partial_{-}M, \partial_{+}M)$ be a Riemannian band. A smooth function $b : [0,T) \to \mathbb{R}$ is said to satisfy the \emph{barrier condition} for the functional \eqref{modified_functional} if, for each connected component $S \subset \partial_{+}M$:
\begin{itemize}
\item $b$ smoothly extends to $S$ and satisfies $H_S\ge b|_S$, where $H_S$ is the mean curvature of $S$ with respect to the outward normal; or
\item $b\to-\infty$ towards $S$.
\end{itemize}
\end{definition}

Gromov first claimed the existence and regularity of a minimizer for $\mu^h$ among Caccioppoli (see Section 5.1 of \cite{Grom}). A  rigorous proof was later established by Zhu \cite[Proposition 2.1]{Zhu2} and by Chodosh and Li \cite[Proposition~12]{CL} given in the following result (in our context).
\begin{proposition}[Existence of $\mu$-bubble]\label{exist}
For a Riemannian band $\left(M^n, g\right)$ with $3 \leq n \leq 7$, if the function $h:[0,T)\to\mathbb{R}$ satisfies the barrier condition, then there exists an $\Omega \in \mathcal{C}$ with smooth boundary such that
$$
\mu^h(\Omega)=\inf _{\Omega^{\prime} \in \mathcal{C}} \mu^h\left(\Omega^{\prime}\right),
$$
where 
$$\mathcal{C}=\{\textrm{Caccioppoli sets }\Omega'\subset M\mbox{ such that  }M\setminus\Omega\Subset\phi^{-1}([0,T))\}.$$ 
\end{proposition}

\section{Proof of Theorem \ref{teoA} and \ref{teoB}}\label{teoab}
We now turn to the proofs of our main Ricci curvature estimates. Using $\mu$-bubble method, we establish Theorem \ref{teoA} and Theorem \ref{teoB} below. 
\subsection{Estimates under Ricci curvature}

A connected manifold with compact boundary is called a half-open band if it has at least one noncompact end. Equivalently, it may be regarded as a band obtained from a compact band by removing one or more of its boundary components.

\begin{proposition}\label{key_estimate}
     Let $( M^{3}, g)$ be an oriented half-open Riemannian band with  $Ric_g \geq -2.$ Assume that $\phi: M\to [0,1)$ is a proper surjective smooth function. 
\begin{itemize}
    \item[a)]  Suppose that $\phi$ satisfies $(\star)_0$-condition and the boundary $\Sigma=\partial_-M$ is connected with mean curvature $H\geq2\coth\left(\frac{\pi}{2}\right).$ Then
      $$\operatorname{Lip}\phi\geq\frac{4}{\pi}.$$
    \item[b)] Fix an integer  ${\mathfrak g}\geq2$ and  a constant $\Lambda$ with $0<\Lambda<2.$ Suppose that $\phi$ satisfies $(\star)_{\mathfrak g-1}$-condition and the boundary $\Sigma=\partial_-M$ is connected with mean curvature $H\geq \Lambda\coth\left(\frac{\pi}{2}\right)$. If moreover  
  $$|\Sigma|\leq \frac{8\pi(\mathfrak g-1)}{4-\Lambda^2},$$
  then 
        $$\operatorname{Lip}\phi \geq \frac{2\Lambda}{\pi}.$$
\end{itemize}
\end{proposition}

\begin{proof}
We give a unified approach for both items. Let $\phi : M \to [0,1)$ be a surjective proper smooth function and assume by contradiction that
$$
\operatorname{Lip}\phi < \frac{2\Lambda}{\pi}, 
$$
where $0<\Lambda\leq2.$ Let
$h:[0,1/\epsilon)\to\mathbb R$ be a function defined by
$$
h(t)=-\Lambda \coth \left(\frac{\pi}{2}(1-\epsilon t)\right),
$$
where $\epsilon>1$ will be fixed later.

Define the prescribed mean curvature functional
$$
\mu(\Omega)=\mathcal{H}_g^{2}\left(\partial^* \Omega\right)-\int_{M \backslash \Omega} h \circ \phi\; \mathrm{d} \mathcal{H}_g^{3}
$$
on the class $\mathcal{C}$ of Caccioppoli sets $\Omega$ in $M$ such that 
$$
M\backslash \Omega \Subset \phi^{-1}([0, 1/ \epsilon)).
$$
Assuming that $\epsilon$ is a regular value of $\phi$, it follows from Proposition \ref{exist} that there exists a smooth minimizer $\hat{\Omega}$.

Since  $\partial \hat{\Omega}$ is homologous to $\phi^{-1}(0)$, $\partial \hat{\Omega}$ has a component $\hat{\Sigma}$ with genus greater than one if $\mathcal{F}^{(\mathfrak g-1\geq 1)}_M\neq\emptyset,$  and  greater than zero if $\mathcal{F}^{(0)}_M\neq\emptyset.$
Moreover, it follows from the first variation of the prescribed mean curvature functional that the mean curvature of $\hat{\Sigma}$ with respect to $\nu$ satisfies $H_{\hat\Sigma}=-h \circ \phi$ on $\hat{\Sigma}.$

The second variation formula \eqref{second} gives, for $\varphi=1$,
$$
0 \leq \int_{\hat{\Sigma}} -\left(\operatorname{Ric}(\nu, \nu)+|A|^2+\nu(h \circ \phi)\right) \, d\mathcal{H}_g^2.
$$
Using
$$
\operatorname{Ric}(\nu, \nu)+|A|^2\geq -4+(h \circ \phi)^2-2K_{\hat\Sigma},
$$
we obtain
$$
\int_{\hat{\Sigma}} 2K_{\hat\Sigma} \, d\mathcal{H}_g^2
\geq \int_{\hat{\Sigma}} -4+(h \circ \phi)^2+\nu(h \circ \phi)\, d\mathcal{H}_g^2.
$$
A direct computation yields
$$
(h \circ \phi)^2+\nu(h \circ \phi)
\geq \Lambda^2 + \left(\Lambda^2-\frac{\pi}{2}\Lambda \epsilon \operatorname{Lip}\phi\right)
\mbox{csch}^2 \left(\frac{\pi}{2}(1-\epsilon\phi)\right).
$$
Since we may assume
$
\pi \epsilon \operatorname{Lip}\phi < 2\Lambda,
$ 
it follows that $
(h \circ \phi)^2+\nu(h \circ \phi) > \Lambda^2,
$
and therefore
$$
4\pi \chi(\hat{\Sigma})
=\int_{\hat{\Sigma}} 2K_{\hat\Sigma}\, d\mathcal{H}_g^2
> \int_{\hat{\Sigma}} (-4+\Lambda^2)\, d\mathcal{H}_g^2.
$$

Now we distinguish the two cases:

\medskip

\noindent
\textbf{ Case $0<\Lambda<2.$} Then 
$
4\pi \chi(\hat{\Sigma}) >(-4+\Lambda^2) |\hat{\Sigma}|,
$
which implies
$$
|\hat{\Sigma}| > \frac{8\pi (g(\hat{\Sigma})-1)}{-4+\Lambda^2}.
$$
Using the comparison
$$
|\hat{\Sigma}| \leq \mathcal{H}_g^{2}(\partial^* \hat\Omega)-\int_{M \backslash \hat \Omega} h \circ \phi\; \mathrm{d} \mathcal{H}_g^{3} \leq |\Sigma|,
$$
we conclude
$$
|\Sigma| > \frac{8\pi (\mathfrak g-1)}{-4+\Lambda^2},
$$
which is a contradiction.
\medskip

\noindent
\textbf{Case  $\Lambda=2$.} Then $
4\pi \chi(\hat{\Sigma}) > 0,
$
which implies $\chi(\hat{\Sigma})>0$, a contradiction since
$g(\hat{\Sigma})\geq 1$.

This completes the proof.
\end{proof}

\subsection{Equality Case}
In this subsection, we prove the rigidity statement in Theorem \ref{teoA}. Fix a constant $\Lambda>0$. 
We now establish a more general rigidity result from which Theorem \ref{teoA} follows by setting $\Lambda=2$.  Assume that 
$$
\mbox{Ric}_g\geq-\Lambda^2/2
$$
and that  the mean curvature satisfies $H\geq\Lambda\coth\left(\frac{\pi}{2}\right)$. Assume also that there exists a function $\phi\in\mathcal{F}^{(0)}_M$ such that 
$$
\operatorname{Lip}\phi=\frac{2\Lambda}{\pi}.
$$

To prove the equality case in Theorem \ref{teoA}, we will need a family of functions $h_\varepsilon$ defined on a subdomain $[0,T_\varepsilon)$ of $[0,1)$ that approximate the limiting function
$$
h(t) = -\Lambda \coth\!\left(\frac{\pi}{2}(1-t)\right),
$$
but which satisfy a crucial sign-changing property for the quantity
$$
Q_\varepsilon(t) := h_\varepsilon(t)^2 + \frac{2\Lambda}{\pi} h_\varepsilon'(t).
$$
This sign change will be essential for localizing the minimizing surfaces within a compact region.

The construction proceeds as follows. Fix $c_0 \in (0,1)$ and define a small perturbation map $\phi_\varepsilon(t):=t+\varepsilon \alpha(t),$ where  $\alpha:[0,1]\to\mathbb{R}$ is given by 
$
\alpha(t):=1-(t-c_0)^2. 
$
Since $\varepsilon<1/2,$ we have  
 $\phi_\varepsilon'(t)>0$ for all $t\in[0,1].$  Moreover, there exists a unique $T_\varepsilon\in(0,1)$ such that
$
\phi_\varepsilon(T_\varepsilon)=1.
$

Define
$$
h_\varepsilon(t):=
-\Lambda
\coth\!\Bigl(
\frac{\pi}{2}\bigl(1-\phi_\varepsilon(t)\bigr)
\Bigr),
\qquad
t\in[0,T_\varepsilon).
$$ 
Note that
$$
\begin{aligned}
h_\varepsilon^2+\frac{2\Lambda}{\pi}h_\varepsilon'
=
\Lambda^2
+
\varepsilon
\frac{2\Lambda}{\pi}
h'(\phi_\varepsilon)
\alpha'(t).
\end{aligned}
$$
As
$
\alpha'(t)=-2(t-c_0),
$
it follows that
$
t<c_0$ implies $h_\varepsilon^2+\frac{2\Lambda}{\pi}h_\varepsilon'
<
\Lambda^2,
$
and when $
c_0<t<T_\varepsilon$ we have $h_\varepsilon^2+\frac{2\Lambda}{\pi}h_\varepsilon'
>
\Lambda^2.
$

Since $\phi_\varepsilon(T_\varepsilon)=1$,
$$
1-T_\varepsilon
=
\varepsilon\alpha(T_\varepsilon),
$$
and therefore $T_\varepsilon\to1$ as $\varepsilon\to0$. 
Finally,
$
\phi_\varepsilon(t)=t+\varepsilon\alpha(t)\to t
$
uniformly on compact subsets of $[0,1)$, hence
$$
h_\varepsilon(t)\to
-\Lambda
\coth\!\Bigl(\frac{\pi}{2}(1-t)\Bigr)
$$
uniformly in compact subsets of $[0,1)$.

These properties are summarized in the following lemma.

\begin{lemma}\label{intersect}
Fix $\Lambda > 0$. For sufficiently small $\varepsilon > 0$, there exists a smooth function $h_\varepsilon : [0, T_\varepsilon) \to \mathbb{R}$, with $T_\varepsilon \nearrow 1$ as $\varepsilon \to 0$, and a constant $c_0 \in (0,1)$ such that $h_\varepsilon' < 0$ everywhere on its domain. The quantity
$$
Q_\varepsilon(t) := h_\varepsilon(t)^2 + \frac{2\Lambda}{\pi} h_\varepsilon'(t)
$$
satisfies $Q_\varepsilon(t) < \Lambda^2$ for $t < c_0$ and $Q_\varepsilon(t) > \Lambda^2$ for $c_0 < t < T_\varepsilon$. Moreover, $h_\varepsilon$ converges uniformly on compact subsets of $[0,1)$ to $h(t) = -\Lambda \coth\!\bigl(\frac{\pi}{2}(1-t)\bigr)$ as $\varepsilon \to 0$.
\end{lemma}

Set $h_k := h_{\epsilon_k}$, where $\{\varepsilon_k\}$ is a sequence of positive numbers converging to $0$ such that $T_{\epsilon_k}$ are regular values of $\phi$. For each $k$,  Proposition \ref{exist} guarantees the existence of a minimizer $\hat{\Omega}_k$ for the functional $\mu^{h_k}$.  
After passing to a subsequence,  $\hat\Omega_k$ converge to a local minimizer $\hat\Omega$ of $\mu^{h}$ with $h(t) = -\Lambda\coth(\tfrac{\pi}{2}(1-t))$.

Inspired by the arguments of \cite[Section 3]{Zhu2}, we prove the following.

\begin{proposition}\label{classification}
Let $(M^3,g)$ be an orientable non‑compact Riemannian manifold with compact boundary $\partial M.$ Suppose $\operatorname{Ric}_g \ge -\Lambda^2/2$ for some $\Lambda>0$, and assume that there exists a function $\phi \in \mathcal{F}_M^{(0)}$ such that
$$
\operatorname{Lip}\phi = \frac{2\Lambda}{\pi}.
$$ Then the following holds.
\begin{itemize}
    \item[a)]\textbf{Structure of the minimizer.} The boundary $\partial\hat{\Omega}$ contains a connected component $\hat{\Sigma}$ contained in $\phi(t_0)$ for 
some $t_0\in[0,c_0],$ which is a flat torus. Moreover, for every unit vector $X\in T_p\hat\Sigma$ we have $\operatorname{Ric}_g(X,X)=-\Lambda^2/2$, and 
$\langle \nabla\phi,\nu\rangle = \frac{2\Lambda}{\pi}$, where $\nu$ is the inward unit normal vector field on $\partial \hat\Omega$.
    \item[b)]\textbf{Foliation near $\hat\Sigma$.} There exists $\delta>0$ and a smooth local foliation $\{\hat\Sigma(s)\}_{s\in(-\delta,\delta)}$ of a neighborhood of $\hat{\Sigma}=\hat{\Sigma}(0)$ such that  each $\hat\Sigma(s)$ is a graph over $\hat\Sigma$ with graph function $u_s$ taken along $\nu$. 
These functions satisfy
    \begin{equation}\label{normalization}
    \frac{\partial u_s}{\partial s}\bigg|_{s=0}  = 1 \quad \text{and} \quad \frac{1}{|\hat \Sigma|}\int_{\hat \Sigma} u_s \, dv_g = s,   
    \end{equation}
    and the quantity $\hat H(s) - h\circ\phi$ is constant on $\hat\Sigma(s)$, where $\hat H(s)$ denotes the mean curvature of $\hat\Sigma(s)$ with respect to the unit normal induced by the foliation.
\item[c)]
 Suppose that, for every $|s|<\delta$, the region bounded by
$\hat{\Sigma}$ and $\hat{\Sigma}(s)$ is disjoint from
$\partial\hat{\Omega}\setminus\hat{\Sigma}$.
Let $\hat{\Omega}(s)$ be the region whose boundary is
$$
\partial\hat{\Omega}(s)
=
\hat{\Sigma}(s)\cup
\bigl(\partial\hat{\Omega}\setminus\hat{\Sigma}\bigr).
$$
Then $\hat{\Omega}(s)$ is a local minimizer of $\mu^{h}$.
\end{itemize}
\end{proposition}

\begin{proof}

We divide the proof into three parts corresponding to items a), b) and c).

\paragraph{Proof of Item a).}
For each $k$, let $\hat\Sigma_k$ be a connected component of $\partial\hat\Omega_k$ of genus $g(\hat\Sigma_k)\ge 1.$  Such a component exists since  $\phi\in\mathcal{F}_M^{(0)}$). 
Applying the second variation formula \eqref{second} to $\hat\Sigma_k$ with  test function $\varphi\equiv 1,$ we obtain
\begin{eqnarray*}
 \int_{\hat{\Sigma}_k} \bigl(-\Lambda^2 + (h_k \circ \phi)^2 + \nu_k(h_k \circ \phi)\bigr)\ d\mathcal{H}_g^2\le  4\pi\chi(\hat{\Sigma}_k),
\end{eqnarray*}
where $\nu_k$ is the outward unit normal to $\hat\Sigma_k,$ and $\chi(\hat \Sigma_k)$ denotes the Euler characteristic of $\hat \Sigma_k$.

By Lemma \ref{intersect}, for any $t\in[0,c_0]$ we have $h_k^2 + \frac{2\Lambda}{\pi}h_k' < \Lambda^2$, while for $t>c_0$ the opposite inequality holds. Since $$
|\langle\nabla\phi,\nu_k\rangle|\le\operatorname{Lip}\phi=2\Lambda/\pi\quad\mbox{ and }\quad h_k'<0,
$$ 
 it follows from Lemma \ref{intersect} that every  $\hat\Sigma_k$ intersects the compact set $$\mathcal K=\phi^{-1}\left(\left[0, c_0\right]\right).$$

Since each $h_k$ is negative,  we have 
$$
|\hat{\Sigma}_k| \leq \mu^{h_k}(\hat{\Omega}_k) \leq |\Sigma|,$$
which implies that the areas of $\hat{\Sigma}_k$ are uniformly bounded. Therefore, the curvature estimates for stable $h_k$-surfaces (see \cite[Theorem 3.6]{XZ}), together with elliptic regularity, imply that after passing to a subsequence the surfaces $\hat\Sigma_k$ converge smoothly and locally graphically to a smooth limit surface $\hat\Sigma$ with $\hat\Sigma\cap\mathcal{K}\neq\emptyset$.

Using the explicit expression for $h_k$ from Lemma 3.2, a direct computation gives
$$
\int_{\hat{\Sigma}_k} \left| \langle \nabla \phi, \nu_k \rangle - \frac{2\Lambda}{\pi} \right| |h_k' \circ \phi| \, d\mathcal{H}_g^2
\leq C \epsilon_k \, |\hat{\Sigma}_k \cap \mathcal{K}|,
$$
which tends to $0$ as $k \to \infty$. Hence, passing to the limit, we obtain
$$
\langle \nabla \phi, \nu \rangle = \frac{2\Lambda}{\pi},
$$
where $\nu$ is the limit of $\nu_k$.  Since $\operatorname{Lip}\phi = \tfrac{2\Lambda}{\pi}$, equality holds in Cauchy–Schwarz and therefore $\nabla\phi=\tfrac{2\Lambda}{\pi}\nu,$ which implies that 
$\nabla_{\hat{\Sigma}}\phi = 0 $  and, thus, for some $t_0 \in [0,c_0],$ $\hat{\Sigma}$ has a closed component $\hat{\Sigma}'$ contained in a level set $\phi^{-1}(t_0)$.  

The $L^1_{\mathrm{loc}}$ convergence $\chi_{\hat\Omega_k}\to\chi_{\hat\Omega}$ implies $\hat\Sigma\subset\partial\hat\Omega$, and therefore $\hat\Sigma'$  is a connected component of $\partial\hat\Omega$. 
Therefore, we may 
choose an open neighborhood $V \subset M$ of $\hat{\Sigma}'$ such that
$$
\hat{\Sigma} \cap V = \hat{\Sigma}'.
$$
By the smooth local convergence $\hat{\Sigma}_k \to \hat{\Sigma}$ and arguing as in \cite[Proposition B.1]{MN}, for $k$ sufficiently large the intersection
$\hat{\Sigma}_k \cap V$ can be written as a finite union of graphs over
$\hat{\Sigma}'$. Since each $\hat{\Sigma}_k$ is connected and
$\hat{\Sigma}_k $ converges to $\hat{\Sigma}'$, the graphical decomposition consists of a single sheet for $k$ sufficiently large. In particular, for $k$ large, $\hat{\Sigma}_k$ is a
graph over $\hat{\Sigma}'$ in $V$. Hence, $\hat \Sigma=\hat \Sigma'$ has genus $g(\hat \Sigma)\geq1.$

Finally, applying the second variation inequality to the limit surface $\hat{\Sigma}$, we get
\begin{eqnarray*}
0&\leq&\int_{\hat{\Sigma}} \bigl(-\Lambda^2 + (h \circ \phi)^2 + \nu(h \circ \phi)\bigr)\,d\mathcal{H}_g^2\le  4\pi\chi(\hat{\Sigma}),
\end{eqnarray*}
Then, we see that all inequalities become equalities. In particular,  $\hat{\Sigma}$ has genus one and   $\mbox{Ric}_g(X,X)=-\Lambda^2/2$ for any unit vector $X\in T_p\hat{\Sigma},$ $d\phi(\nu) = 2\Lambda/\pi$ and  $\hat\Sigma$ is flat. This completes the proof of Item a).

Item b) follows from a standard application of the implicit function theorem, provided  the stability operator (or Jacobi operator) associated to the stability inequality \eqref{second} reduces to $L=-\Delta_{\hat \Sigma},$ see, for instance, \cite{BBN,N}

It remains to prove Item c). Let $s \in [0,\delta)$. We first show that $\hat{H}(\tau) - h \circ \phi \equiv 0$ for every $0 < \tau < s$. Suppose, to the contrary, that for some $\tau$ this constant is positive (a negative value is impossible, since $\hat{\Omega}$ is a local minimizer and one could deform inward to decrease energy).

Consider the ordinary differential equation
$$
\bar h_\delta^2 + \frac{2\Lambda}{\pi}\bar h_\delta' = \Lambda^2 + \delta,\qquad \bar h_\delta(t_0) = h(t_0)+\delta,
$$
with $\delta>0$ small. 
Choosing $\delta$ sufficiently small, we have $\hat H(0)-\bar h_\delta\circ\phi= h\circ\phi-\bar h_\delta\circ\phi> 0$ on $\hat\Sigma$ and $\hat H(\tau) > \bar h_\delta\circ\phi$ on $\hat\Sigma(\tau)$. 
Now consider the functional $\mu^{\bar h_\delta}$ in the region bounded by $\hat\Sigma$ and $\hat\Sigma(\tau)$. 

By Proposition \ref{exist} there exists a smooth minimizer $\tilde\Omega$ with $\hat\Sigma\subset\partial\tilde\Omega$. 
Because the projection of $\partial\tilde\Omega\setminus\hat\Sigma$ onto $\hat\Sigma$ has degree one,  $\partial\tilde\Omega$ must contain a component $\tilde\Sigma$ of genus one. Applying the second variation to $\tilde\Sigma$ yields
$$
0 \ge 4\pi\chi(\tilde\Sigma) \ge  \delta\,\operatorname{area}(\tilde\Sigma) > 0,
$$
a contradiction. 
Hence $\tilde H(\tau) -h\circ\phi = 0$ for all $\tau\in(0,s)$.

Now the remainder of the argument follows as in \cite[Proposition 3.5]{Zhu2}.  
\end{proof}

\subsection{Proof of the Ricci curvature results and its consequences}

\begin{proof}[Proof of Theorem \ref{teoA}]
The lower bound $\operatorname{Lip}\phi \ge 4/\pi$ follows directly from Item  a) of Proposition \ref{key_estimate} with $\Lambda=2$. Assume now that equality holds, namely $\operatorname{Lip}\phi = 4/\pi.$ 

By Proposition \ref{classification} there exists a connected component  $\hat\Sigma\subset\partial\hat\Omega$ which is a flat torus contained in a level set of $\phi$. Moreover, $\langle\nabla\phi,\nu\rangle = 4/\pi$ along $\hat\Sigma$, and $\operatorname{Ric}_g(X,X)=-2$ for every unit vector $X$ tangent to $\hat\Sigma$. 
Furthermore, there exists $\delta>0$ and a smooth foliation
$\{\hat\Sigma(s)\}_{s\in(-\delta,\delta)}$ of a neighborhoodof $\hat\Sigma=\hat\Sigma(0)$ such that each $\hat\Sigma(s)$ is a graph over $\hat\Sigma$ with constant function $ \hat H(s) - h\circ\phi$ on $\hat \Sigma(s)$,
and each associated region $\hat\Omega(s)$ is a
local minimizer of the functional $\mu^h$.

Let $f$ denote the lapse function of the foliation, so that locally the metric
can be written as
$$
g =  f(s)^2\,ds^2 + \hat g(s),
$$
where $g_s$ is the induced metric on $\hat\Sigma_s$.
Since equality holds in the stability inequality, the Jacobi operator of
$\hat\Sigma(s)$ reduces to $L=-\Delta_{\hat\Sigma(s)}.$ The lapse function satisfies $Lf=0.$ Because each leaf $\hat\Sigma_s$ is compact, it follows that $f$ is constant on
$\hat\Sigma_s$. Using the normalization condition from
Proposition \ref{classification}, we conclude that $f\equiv 1$. 
Hence the foliation is given by unit-speed
normal geodesics and $g=ds^2+g_s.$

If $\hat\Sigma\cap\partial M\neq\emptyset$, then the maximum principle implies
that $\hat\Sigma=\partial M$, and therefore the splitting holds on a collar
neighborhood of the boundary. Otherwise, $\hat\Sigma$ lies in the interior and
the splitting holds on a tubular neighborhood of $\hat\Sigma$. The maximal neighborhood of $\hat\Sigma$ on which the metric splits as $(\hat \Sigma\times I , ds^2 + \hat g(t))$ is open and closed in $M.$ By connectedness, $I$ must be of the form $[0,\varepsilon).$

Using item c) of Proposition \ref{classification} and arguing exactly as in \cite[Theorem 1.4]{Zhu2}, we can extend the foliation to a maximal family
$\{\hat\Sigma(s)\}_{0\le s<b}$
such that 
$
 \lim _{s \rightarrow b^{-}} \phi(s)=1.
$

Because $\langle\nabla\phi,\nu\rangle = 4/\pi$ and $\phi$ is constant on each $\hat\Sigma(s)$, we can set $\rho = \pi\phi/4$. 
Then $\rho$ is a distance function (since $|\nabla\rho|=1$) and the level sets $\{\rho = \text{constant}\}$ coincide with the foliation $\{\hat\Sigma(s)\}$. 
Hence we can write
$$
M = \hat\Sigma \times [0,\tfrac{\pi}{4}),\qquad g = \bar g(\rho)+ d\rho^2,
$$
where $\bar g(\rho)$ is the induced metric on the level set $\{\rho=\text{constant}\}$.

We now   show that $g$ has constant sectional curvature $-1$. 

From Proposition \ref{classification} we know that $\operatorname{Ric}_g(X,X) = -2$ for every unit vector $X$ tangent to the level sets. 
Let $\nu=\partial_\rho$,  and let $\{e_1, e_2\}$ be a local orthonormal frame  tangent to the level sets. Fix $p\in M$, since $\mbox{Ric}_g(X,X)=-2$ for every unit vector $X\in T_p\hat\Sigma$, we can show that $R(X,\partial_\rho,\partial_\rho,X)$ does not depend on $X\in T_p\hat\Sigma$ orthogonal to $\partial_p$. Let $W\in T_p\hat\Sigma$ orthogonal both to $X$ and $\partial_\rho$. Then
$$
2R(X ,\partial_\rho,\partial_\rho,X)=R(X ,\partial_\rho,\partial_\rho,X)+R(W ,\partial_\rho,\partial_\rho,W)= \mbox{Ric}_g(\partial_\rho,\partial_\rho)\geq-2.
$$

Consider now the Riccati equation
$$
\partial_\rho\nabla^2\rho+(\nabla^2\rho)^2+R(\cdot,\nabla\rho)\nabla\rho=0,
$$
where $\nabla^2\rho$ is an endomorphism of $T^\star M^3$. Denote by $\lambda(s)$ the least eigenvalue of $\nabla^2\rho$ at an integral curve of $\gamma(s).$ Then 
\begin{equation}\label{EDO}
\lambda_i'(s)+\lambda^2_i(s)\leq 1.    
\end{equation}
Hence for $s\in[0,\pi/4)$, 
\begin{equation}\label{compar}
   \lambda(s)\leq\frac{a\coth (s)-1}{\coth (s)-a}, 
\end{equation}
where $a=\lambda(0).$  Since the foliation exists on the maximal interval $[0,\pi/4),$  we can conclude that one principal curvatures of $\hat\Sigma$ is $-\coth(\pi/4)$. Moreover, the fact that $\lambda_1(0)+\lambda_2(0) = -2\coth(\pi/2),$ we have that the other principal curvature is $-\tanh(\pi/4).$ Note that we also have equality in \eqref{compar}, and thus a principal curvature of  $\hat\Sigma_s=\{\rho=s\}$ is $-\coth\left(\frac{\pi}{4} - s\right)$.  One can prove that the largest eigenvalue of $\nabla^2\rho$ is 
$$
\lambda_{max}(s) \leq \tanh\left(\frac{\pi}{4} -s\right).
$$
Using that each Gauss curvature of $\Sigma_s$ is zero and Gauss equation we conclude that equality holds and, as a consequence, the sectional curvature of $g$ must be $-1.$

Because the torus $\hat\Sigma_\rho$ is flat, we have $\lambda_1(\rho)\lambda_2(\rho) = 1$ and $\lambda_1(\rho)+\lambda_2(\rho) = -2\coth(\pi/2 - 2\rho).$ Solving these equations yields

$$\lambda_1(\rho) = -\tanh\!\left(\frac{\pi}{4} -\rho\right),\qquad
\lambda_2 (\rho) = -\coth\!\left(\frac{\pi}{4} -\rho\right).$$
Since $\partial_\rho g_\rho = 2A_\rho,$ where $A_\rho=\nabla^2\rho|_{T\hat{\Sigma}(\rho)}$, the metric evolves according to $\partial_\rho g_\rho(e_i,e_i) = 2\lambda_i(\rho)g_\rho(e_i,e_i).$ Choosing local principal coordinates $(s_1,s_2)$, integration yields
$$
g_\rho = e^{2\int \lambda_1(\rho)\, d\rho} \, ds_1^2
+ e^{2\int \lambda_2(\rho)\, d\rho} \, ds_2^2.
$$
A direct computation yields
$$
g = d\rho^2
+ \cosh^2\!\left(\tfrac{\pi}{4} - \rho\right)\, ds_1^2
+ \sinh^2\!\left(\tfrac{\pi}{4} - \rho\right)\, ds_2^2.
$$

Since each leaf $\Sigma_\rho$ is a flat torus, its universal covering is
$\mathbb R^2$. Passing to the universal covering, we obtain
$$
M_0=\mathbb R^2\times[0,\tfrac{\pi}{4})
$$
equipped with the metric above. The deck transformations act by Euclidean
isometries on the $\mathbb R^2$-factor, preserving the warped product
structure. Hence $(M,g)$ is isometric to a quotient of this doubly warped
product.

Finally, since $\rho=\pi/4\phi,$ the function $\phi$ coincides, up to scaling and translation, with the signed distance function. This completes the proof.
\end{proof}

\begin{proof}[Proof of Theorem \ref{teoB}]
   The result follows immediately from item $b)$ of Proposition \ref{key_estimate} with $\Lambda=\sqrt{4-2c}.$
\end{proof}

The following result can be found in \cite{Zhu} and \cite{CZ}.

\begin{lemma}\label{Zhu_lemma}
Let $(M,g)$ be a smooth Riemannian manifold with boundary
$\partial M = \partial_- M \cup \partial_+ M$.
Assume that
$$
\textrm{width}(M,g)>2l
$$
Then there exists a smooth surjective map
$$
\phi : M \to [-l,l]
$$
such that
\begin{enumerate}
  \item $\phi^{-1}(-l) = \partial_- M$,
  \item $\phi^{-1}(l) = \partial_+ M$,
  \item $|d\phi| \le c < 1$ on $M$ for some constant $c$.
\end{enumerate}
\end{lemma}

\begin{proof}[Proof of Corollary \ref{princ_0}]
 Assume by contradiction that   
 $$
 \textrm{width}(M,g)> \frac{\pi}{2\sqrt{4-2c}}.
 $$ Set $l=\frac{\pi}{4\sqrt{4-2c}}.$  It follows from Lemma \ref{Zhu_lemma} that there exists a smooth surjective 
function $\tilde \phi : M \to \left[-l,l\right],$  with   $\operatorname{Lip}\tilde\phi < 1$ and such that 
$\tilde\phi^{-1}(-l) =\partial_- M$ and $\tilde\phi^{-1}(l) = \partial_+ M$. 
Now define
$$
\phi = \frac{2\sqrt{4-2c}}{\pi}(\tilde\phi+l).
$$ 
A direct verification shows that $\phi : \hat M \to [0,1]$ is surjective, with 
$\phi^{-1}(0) = \partial_- M$ and $\phi^{-1}(1) = \partial_+ M$. 
Moreover,
$$
\operatorname{Lip} \phi 
< \frac{2\sqrt{4 - 2c}}{\pi}.
$$
Let $\hat{M}$ be the half-open band obtained from $M$ by removing $\partial_+ M$. Considering the restriction  $\phi : \hat{M} \to [0,1),$ it follows from Proposition \ref{key_estimate} that  
$$
|\partial_- M| > \frac{8\pi(g(\Sigma) - 1)}{4 - \Lambda^2}
= \frac{4\pi(g(\Sigma) - 1)}{c}.$$
This contradicts the assumed area bound,  and we conclude that
$$
\operatorname{width}(M,g) \leq \frac{\pi}{2\sqrt{4 - 2c}}.
$$
\end{proof}

The proof of Corollary \ref{princ_00} is analogous to that of the previous corollary and will be omitted.

\section{Scalar curvature: Complete case}\label{scalar}

The bulk of this section is dedicated to proving results involving the scalar curvature, culminating in a rigidity theorem under the assumptions of bounded geometry and scalar curvature bounded below. We begin with the following proposition

\begin{proposition}

Let $\Sigma$ be a closed surface of genus $g(\Sigma) \geq 2$ and assume that $M= \Sigma\times[-1,1]$ admits a metric $g$ with scalar curvature $R_g \geq -6.$ Given $c \in (0,2),$ suppose that $H^{\partial_- M}_g\geq c\cdot\coth(\tfrac{\pi}{2})$ and the area of $\partial_-M$ satisfies 
    $$
    |\partial_-M|\leq \frac{4\pi(g(\Sigma)-1)}{3-\frac{3}{4}c^2}.
    $$
Then
$$\textrm{width}(M,g)\leq\frac{\pi}{2c}.$$
\end{proposition}
    
\begin{proof}[Sketch of the proof]
    Since the proof follows the same line as Proposition \ref{key_estimate}, we merely sketch the proof. Assume by contradiction that  there exists a smooth proper surjective function $\phi : M \to [0,1)$ such that 
$
\operatorname{Lip}\phi < \frac{2c}{\pi}.
$
Define 
$h:[0,1/\epsilon)\to\mathbb R$  by
$$
h(t)=-c\cdot\coth \left(\frac{\pi}{2}(1-\epsilon t)\right),
$$
where $\epsilon>1$ will be fixed later. Again, it follows from the second variation of the corresponding prescribed mean curvature functional, and the Gauss equation
$$
\int_{\hat{\Sigma}} 2K_{\hat\Sigma} \, d\mathcal{H}_g^2
\geq \int_{\hat{\Sigma}} -6+\frac{3c^2}{2} + \left(\frac{3c^2}{2}-c\pi\epsilon \operatorname{Lip}\phi\right)
\mbox{csch}^2 \left(\frac{\pi}{2}(1-\epsilon\phi)\right)\, d\mathcal{H}_g^2.
$$
Since we may assume
$
\pi \epsilon \operatorname{Lip}\phi < 2c,
$
it follows that
$$
|\partial_- M| > \frac{8\pi (\mathfrak g-1)}{6-\frac{3}{2}c^2},
$$
which is a contradiction. Arguing as in the proof of Corollary \ref{princ_0}, we can conclude that 
$$\text{width}(M,g)\leq \frac{\pi}{2c}.$$
\end{proof}

Although weaker than the bound in Theorem \ref{teoB}, it may be rigid as $c$ goes to $0$. However, to obtain a rigidity result in the scalar curvature setting we need additional control on the geometry, namely  bounded geometry (uniform bounds on the curvature and injectivity radius). 

The following theorem plays a crucial role in the proof.

\begin{theorem}[K. Frensel\cite{K}]\label{Katia}
    Let $M^m$ be a complete, noncompact manifold and let $x:M^m\to N^n$ be an immersion with mean curvature vector field bounded in norm. Assume that $N$ has bounded geometry. Then the volume of $M$ in the induced metric is infinite.
\end{theorem}

The idea of a proof of the above result consists of showing a bound from below to the volume of a geodesic ball in M, which, together with the covering argument, gives the result. 

\begin{rem}
When $N$ has nonpositive sectional curvature, the aforementioned result  has been proved by Hoffman and Schoen (in an unpublished work).
\end{rem}

\subsection{Proof  of Theorem \ref{princ}}

We first prove the area estimate.

\begin{proposition}\label{prop_help}
       Let $(M^{3}, g)$ be an orientable complete Riemannian manifold with connected and mean-convex boundary $\partial M$. Assume that $(M,g)$ has  bounded geometry,  scalar curvature $R_g \geq -6$ and that  $H_2(M; \mathbb Z)$  is non‑trivial and contains no spherical and toroidal classes. 
Then
     \begin{equation}\label{ineq_thm3}
         |\partial M| \geq \frac{4\pi}{3} (\mathfrak g_0-1),
     \end{equation}
 where $ \mathfrak g_0$ denotes the minimum genus of a connected, closed, homologically nontrivial surface embedded in $M$.
\end{proposition}
  \begin{proof}
  As in \cite{Zhu},  by smoothing the signed distance function, there is a smooth proper function $\phi: M\to [0, +\infty)$,  satisfying $\phi^{-1}(0) = \partial M$ and $\mbox{Lip }\phi < 1.$  

Following \cite[Lemma 2.3]{Zhu}. We define $
h_\varepsilon^-(t) : \left(-\infty, \tfrac{1}{3\varepsilon}\right)
\to \mathbb{R}$ given by $
h_\varepsilon^-(t) = -2\varepsilon \coth\!\left(-\frac{3}{2}\varepsilon t + \tfrac{1}{2}\right).
$
A direct computation shows that $ h_\varepsilon^-$ solves
$$
\frac{3}{2}h_\varepsilon^-(t)^2 + 2h_\varepsilon^-(t)' = 6\varepsilon^2.
$$
Fix a nonnegative smooth function $\eta$ with compact support  in
$[0,\tfrac{1}{6})$ such that $0 \le \eta \le 1$ and $\eta \equiv 1$
near $t=0$. Define
$h_\varepsilon(t) = (1-\eta(t))\, h^-_\varepsilon(t).$
Then $h_\varepsilon$ is smooth on $\left[0,\tfrac{1}{3\varepsilon}\right)$
and coincides with $h_\varepsilon^-$ on
$[\tfrac{1}{6},\tfrac{1}{3\varepsilon})$, which yields

\begin{itemize}
    \item $
\frac{3}{2}h_\varepsilon^2 + 2h_\varepsilon' = 6\varepsilon^2
\quad \text{on } \left[\tfrac{1}{6}, \tfrac{1}{3\varepsilon}\right);
$
    \item On the remaining region, since $h_\varepsilon = O(\varepsilon)$ and
$h_\varepsilon' = O(\varepsilon^2)$ uniformly on compact subsets, we obtain
$$
\left|
\frac{3}{2}h_\varepsilon^2 + 2h_\varepsilon'
\right|
\le C\varepsilon^2
$$
for some universal constant $C>0$.
    \item $h_\varepsilon' < 0$;
    \item As $\varepsilon \rightarrow 0, h_\varepsilon$ converges smoothly to the zero function on any closed interval
    \item   $
\lim _{t \rightarrow \frac{1}{3 \varepsilon}} h_\varepsilon(t)=-\infty.$

\end{itemize}

$$
\frac{3}{2}h_\varepsilon^2 + 2h_\varepsilon' = 6\varepsilon^2
\quad \text{on } \left[\tfrac{1}{6}, \tfrac{1}{3\varepsilon}\right).
$$

For each $\varepsilon>0,$ consider the prescribed mean curvature functional
$$
\mu^\varepsilon(\Omega)=\mathcal{H}_g^{2}\left(\partial \Omega\right)-\int_{M \backslash \Omega} h_{\varepsilon} \circ \phi\; \mathrm{d} \mathcal{H}_g^{3}
$$
in the class of the Caccippoli set $\mathcal{C}_\varepsilon$ in $M$ with reduced boundary $\partial^* \Omega$ such that 
$$
M\backslash \Omega \Subset \phi^{-1}([0, 1/3 \varepsilon]).
$$
It follows from Sard's Theorem that there exists a sequence $\varepsilon_k \rightarrow 0$ as $k \rightarrow \infty$ such that $1/3\varepsilon_k$ are regular values of $\phi.$ Since each function $h_{\varepsilon_k}$ satisfies the barrier condition, Proposition \ref{exist} guarantees the existence of a smooth minimizer $\Omega_k$ in $\mathcal{C}_{\varepsilon_k}$ for functional $\mu^{k}:=\mu^{\varepsilon_k}.$

Observe that $\partial \Omega_k$ is homologous to $\phi^{-1}(0),$ then  $\partial \Omega_k$ has a component, denoted by $\Sigma_k$, whose genus is at least two.  It follows from the second variation formula \eqref{second} and the Gauss equation that 
\begin{eqnarray*}
 \int_{\Sigma_k} |\nabla\varphi|^{2}\mathrm{d}\mathcal{H}_{g}^{2}
 &\geq&\int_{\Sigma_{k}}(\mbox{Ric}_{g_k}(\nu_{g_k},\nu_{g_k})+|A_{g_k}|^2+\langle \nabla h_{\varepsilon_k}\circ\phi, \nu_{g_k} \rangle )\varphi^{2} \, \mathrm{d}\mathcal{H}_{g}^{2}\\
 &\geq&\frac{1}{2}\int_{\Sigma_{k}}(-R_{\Sigma_k}+|A_{g_k}|^2-H_{g_k}^2 +R_g + 2\langle \nabla h_{\varepsilon_k}\circ\phi, \nu_{g_k} \rangle )\varphi^{2} \, \mathrm{d}\mathcal{H}_{g}^{2}\\
& \geq & \frac{1}{2}\int_{\Sigma_{k}}\left(-R_{\Sigma_k}+ \left|A_{g_k}-\frac{H_{g_k}}{2} g_k\right|^2\right)\varphi^2\, \mathrm{d}\mathcal{H}_{g}^{2}\\
& &+\frac{1}{2}\int_{\Sigma_{\varepsilon}}\left(R_g+\left(\frac{3}{2} h_{\varepsilon_k}^2+2 h_{\varepsilon_k}^{\prime}\right) \circ \phi\right)\varphi^2\, \mathrm{d}\mathcal{H}_{g}^{2} 
\end{eqnarray*}
for any non-zero  $\varphi \in C_{0}^{\infty}(\Sigma_{k}),$ where we  used 
$$
\langle \nabla h_{\varepsilon_k}\circ\phi, \nu_{g_k} \rangle\leq |dh_{\varepsilon_k}|\circ\phi\quad\mbox{ and }\quad|A_{\hat g_k}|^2\geq H^2_{\hat g_k}/2.
$$ 
Here, $g_k$ is the induced metric on $\Sigma_k,$   $A_{g_k}$ is the second fundamental form of $\Sigma_k$ in $M,$ $\nu_{g_k}$ is the outward unit normal vector field of $\Sigma_k$ and $R_{\Sigma_k}$ is the scalar curvature of $\Sigma_k.$ 

Taking $\varphi\equiv 1,$ the fact that $R_g\geq-6$ and the Gauss-Bonnet Theorem, we have
\begin{eqnarray*}
      4\pi(1-g(\Sigma_k))&\geq&\frac{1}{2}\int_{\Sigma_{\varepsilon}}(R_g-C\varepsilon_k)\, \mathrm{d}\mathcal{H}_{g}^{2}\\
      &\geq&-3|\Sigma_k|-\frac{1}{2}C\varepsilon_k|\Sigma_k|.
\end{eqnarray*}
Then we obtain
\begin{eqnarray*}
    3|\Sigma_k|&\geq&
   4\pi(g(\Sigma_k)-1)-\frac{1}{2}C\varepsilon_k|\Sigma_k|\\
   &\geq& 4\pi(\mathfrak g_0-1)-\frac{1}{2}C\varepsilon_k|\Sigma_k|.
\end{eqnarray*}
Since, by a simple comparison, we have
\begin{eqnarray*}
   |\Sigma_k|\leq\mathcal{H}_g^{2}(\partial\Omega_k)-\int_{M \backslash \Omega_k} h_k \circ \phi\; \mathrm{d} \mathcal{H}_g^{3} \leq|\partial M|,
\end{eqnarray*}
we obtain, letting $k\to\infty,$ that 
\begin{equation}\label{mean}
     |\partial M|\geq \frac{4\pi}{3}(\mathfrak g_0-1).
\end{equation}
\end{proof}

\begin{proposition}
If $\partial M$ attains equality in \eqref{ineq_thm3}, then  $(M, g)$ is covered isometrically by the cylinder $(\partial M\times[0,+\infty),g|_{\partial M}+dt^2),$ where $\partial M$ has constant Gauss curvature equal to $-3.$ Moreover, $\mathfrak g_0$ coincides with the genus of $\partial M.$
\end{proposition}
 
\begin{proof}
    Let $\{\Sigma_k\}$ be as in the proof of Proposition \ref{prop_help}.  We claim that every $\Sigma_k$ must intersect the compact set $\mathcal K=\phi^{-1}\left(\left[0, \frac{1}{6}\right]\right).$  Indeed,  if $\mathcal{K}\cap\Sigma_k=\emptyset,$  then repeating the previous argument yields
$$
(3-3\varepsilon_k^2)|\partial M|\geq(3-3\varepsilon_k^2)|\Sigma_k|\geq
   4\pi(\mathfrak g_0-1),
$$
which contradicts the equality assumption. 

We also observe that the equality in \eqref{ineq_thm3} implies that all $\Sigma_k$ have genus $\mathfrak g_0$. Since the hypersurfaces $\Sigma_k$ are $h_{\varepsilon_k}$-minimizing boundaries with uniform bounded area, we can invoke the curvature estimates in \cite[Theorem 3.6]{XZ}. Thus, after passing to a subsequence, $ \Sigma_k $ converges smoothly in a locally graphical sense with multiplicity one to an area-minimizing boundary $\Sigma'.$ 

We claim that $\Sigma'$ is compact. Indeed,  if $\Sigma'$ were noncompact, then by Theorem \ref{Katia}  its area would be infinite, contradicting the inequality  $|\Sigma'|\leq|\partial M|$. Moreover, since $\Sigma_k$ is connected, $\Sigma'$ is a surface of genus $\mathfrak g_0$  and $\Sigma_k$  becomes a graph over $\Sigma,$ for $k$ sufficiently large. Moreover, by stability
$$
\frac{4\pi}{3}(\mathfrak g_0-1)=|\partial M|\geq|\Sigma'|\geq\frac{4\pi}{3}(\mathfrak g_0-1).
$$ 
Therefore, $\Sigma'$ is area-minimizing and attains the equality case.  Equality in the stability inequality implies that the limit surface is
totally geodesic and satisfies $\operatorname{Ric}(\nu,\nu)=0$ on $\Sigma'$. We can argue as in Nunes \cite[Theorem 5]{N}, where standard splitting arguments then show that a neighborhood of the surface splits
isometrically.
  \end{proof}

\end{document}